\documentclass[11 pt]{article}
\usepackage{amsmath,amssymb,amsfonts,amsthm,graphicx,cite,epsfig,makecell, array,mathrsfs}
\usepackage[varg]{pxfonts}
\usepackage{wrapfig}
\usepackage{enumitem}
\usepackage{epstopdf,footnote,multirow,booktabs,bigints,pstool,booktabs,caption,subcaption ,lipsum}
\usepackage[table]{xcolor}
\usepackage[colorlinks = true,
            linkcolor = red,
            citecolor = blue]{hyperref}
\usepackage[symbol]{footmisc}
\usepackage{titlesec, float}
\usepackage[export]{adjustbox}
\theoremstyle{plain}
\newtheorem{thm}{Theorem}[section]

\theoremstyle{definition}
\newtheorem{de}[thm]{Definition}
\newtheorem{co}[thm]{Corollary}

\theoremstyle{remark}
\newtheorem{re}{\sc \textbf{Remark}}

\numberwithin{equation}{section}
\newcolumntype{?}{!{\vrule width 1.4pt}}

\def\correspondingauthor{\footnote{Corresponding author.}}
\titlelabel{\thetitle.\quad}
\renewenvironment{abstract}
               {\list{}{\rightmargin\leftmargin}%
                \item[\textbf{\hspace{8.6mm}Abstract ---}]\relax}
               {\endlist}
\DeclareUrlCommand{\url}{%
    \def\UrlLeft##1\UrlRight{\underline{##1}}}

\date{}
\title{A Comprehensive Subclass of Bi-Univalent Functions Associated with Chebyshev Polynomials of the Second Kind}
\author{Feras Yousef$^{1,}$\correspondingauthor{}\,\,,\, Somaia Alroud$^{2}$,  Mohamed Illafe$^{3}$ \vspace{0.1in}\\
\footnotesize{$^{1,2}$Department of Mathematics, The University of Jordan, Amman 11942, Jordan. }  \\
 \footnotesize{$^{3}$School of Engineering, Math, $\&$ Technology, Navajo Technical University, Crownpoint, NM 87313, USA. }  \\
 \footnotesize{e-mail: $^1$fyousef@ju.edu.jo}, \,$^2$somaiaroud41@gmail.com,\, $^3$millafe@navajotech.edu}
\begin{document}
\maketitle
\begin{abstract} Our objective in this paper is to introduce and investigate a newly-constructed subclass of normalized analytic and bi-univalent functions by means of the Chebyshev polynomials of the second kind. Upper bounds for the second and third Taylor-Maclaurin coefficients, and also Fekete-Szeg\"{o} inequalities of functions belonging to this subclass are founded. Several connections to some of the earlier known results are also pointed out. \\
\end{abstract}



\section{Introduction, Definitions and Notations}
Let $\mathcal{A}$ denote the class of all analytic functions $f$ defined in the open unit disk $\mathbb{U}=\{z\in\mathbb{C}:\left\vert z\right\vert <1\}$ and normalized by the condition $f(0)= f^{\prime}(0)-1=0$. Thus each $f\in\mathcal{A}$ has a Taylor-Maclaurin series expansion of the form:
\begin{equation} \label{ieq1}
f(z)=z+\sum\limits_{n=2}^{\infty}a_{n}z^{n}, \ \  (z \in\mathbb{U}).
\end{equation}

Further, let $\mathcal{S}$ denote the class of all functions $f \in\mathcal{A}$ which are univalent in $\mathbb{U}$ (for details, see \cite{Duren}; see also some of the recent investigations \cite{L,C1,C4,C5,C6}).

Given two functions $f$ and $g$ in $\mathcal{A}$. The function $f$ is said to be subordinate to $g$ in $\mathbb{U}$, written as $f(z)$ $\prec$
$g(z)$, if there exists a Schwarz function $\omega(z)$, analytic in $\mathbb{U}$, with
\begin{center}
$\omega(0) = 0$ and $\left\vert \omega(z)\right\vert <1$ for all $z \in\mathbb{U}$,
\end{center}
such that $f(z)=g\left(\omega(z)\right)$ for all $z \in\mathbb{U}$. Furthermore, if the function $g$ is univalent in $\mathbb{U}$, then we have the following equivalence (see \cite{mill} and \cite{C2}):%
\[
f(z)\prec g(z)\Leftrightarrow f(0)=g(0)\text{ and }f(\mathbb{U})\subset
g(\mathbb{U}).
\]

Two of the important and well-investigated subclasses of the analytic and univalent function class $\mathcal{S}$ are the class
$\mathcal{S}^{\ast}(\alpha)$ of starlike functions of order $\alpha$ in $\mathbb{U}$ and the class $\mathcal{K}(\alpha)$ of convex functions of order $\alpha$ in $\mathbb{U}$. By definition, we have

\begin{equation}
\mathcal{S}^{\ast}(\alpha):=\left\{ f: \ f \in \mathcal{S} \ \ \text{and} \ \ \mbox{Re}\left\{ \frac{zf^{\prime }(z)}{f(z)}\right\} >\alpha,\quad (z\in \mathbb{%
U}; 0\leq \alpha <1) \right\},  \label{d1}
\end{equation}%
and
\begin{equation}
\mathcal{K}(\alpha):=\left\{ f: \ f \in \mathcal{S} \ \ \text{and} \ \ \mbox{Re}\left\{ 1+\frac{zf^{\prime \prime }(z)}{f^{\prime }(z)}\right\} >\alpha,\quad (z\in \mathbb{%
U}; 0\leq \alpha <1) \right\}. \label{d2}
\end{equation}%

It is clear from the definitions (\ref{d1}) and (\ref{d2}) that $\mathcal{K}(\alpha) \subset \mathcal{S}^{\ast}(\alpha)$. Also we have

\begin{equation*}
f(z) \in \mathcal{K}(\alpha) \ \ \text{iff} \ \ zf^{\prime}(z) \in \mathcal{S}^{\ast}(\alpha),
\end{equation*}%
and
\begin{equation*}
f(z) \in \mathcal{S}^{\ast}(\alpha) \ \ \text{iff} \ \ \int_{0}^{z} \frac{f(t)}{t} dt =F(z) \in \mathcal{K}(\alpha).
\end{equation*}%

It is well-known \cite{Duren} that every function $f\in \mathcal{S}$ has an inverse map $f^{-1}$ that satisfies the following conditions:
\begin{center}
$f^{-1}(f(z))=z \ \ \ (z\in \mathbb{U}),$
\end{center}
and
\begin{center}
$f\left(f^{-1}(w)\right)=w$ $\ \ \ \left( |w|<r_{0}(f);r_{0}(f)\geq\frac{1}{4}\right)$.
\end{center}

In fact, the inverse function is given by
\begin{equation} \label{ieq2}
f^{-1}(w)=w-a_{2}w^{2}+(2a_{2}^{2}-a_{3})w^{3}-(5a_{2}^{3}-5a_{2}a_{3}+a_{4})w^{4}+\cdots.
\end{equation}

A function $f\in \mathcal{A}$ is said to be bi-univalent in $\mathbb{U}$ if both $f(z)$ and $f^{-1}(z)$ are univalent in $\mathbb{U}$. Let $\Sigma$ denote the class of bi-univalent functions in $\mathbb{U}$ given by (\ref{ieq1}). For a brief history and some interesting examples of functions and characterization of the class $\Sigma$, see Srivastava et al. \cite{C27}, Frasin and Aouf \cite{C28}, and Magesh and Yamini \cite{F3}.

In 1967, Lewin \cite{C21} investigated the bi-univalent function class $\Sigma $ and showed that $|a_{2}|<1.51$. Subsequently, Brannan and Clunie \cite{C22} conjectured that  $|a_{2}|\leq \sqrt{2}.$ Later, Netanyahu \cite{C23} showed that $\max$ $|a_{2}|=\frac{4}{3}$ if $f\in \Sigma.$ Brannan and Taha \cite{C25} introduced certain subclasses of a bi-univalent function class $\Sigma$ similar to the familiar subclasses $\mathcal{S}^{\ast}(\alpha)$ and $\mathcal{K}(\alpha)$ of starlike and convex functions of order $\alpha$ ($0\leq \alpha <1$), respectively (see \cite{C26}). Thus, following the works of Brannan and Taha \cite{C25}, for $0\leq \alpha <1,$ a function $f\in \Sigma $ is in the class $\mathcal{S}_{\Sigma}^{\ast }\left( \alpha \right) $ of bi-starlike functions of order $\alpha$;
or $\mathcal{K}_{\Sigma }\left( \alpha \right)$ of bi-convex functions of order $\alpha$ if both $f$ and $f^{-1}$ are respectively starlike or convex
functions of order $\alpha.$ Recently, many researchers have introduced and investigated several interesting subclasses of the bi-univalent function class $\Sigma$ and they have found non-sharp estimates on the first two Taylor-Maclaurin coefficients $|a_{2}|$ and $|a_{3}|$. In fact, the aforecited work of Srivastava et al. \cite{C27} essentially revived the investigation of various subclasses of the bi-univalent function class $\Sigma$ in recent years; it was followed by such works as those by Frasin and Aouf \cite{C28}, Xu et al. \cite{C210}, \c{C}a\u{g}lar et al. \cite{C29}, and others (see, for example, \cite{C211,C212,F4,F1} and \cite{C213}). The coefficient estimate problem for each of the following Taylor-Maclaurin coefficients $|a_{n}|$ $(n\in \mathbb{N}\backslash \{1,2\})$ for each $f\in \Sigma$  given by (\ref{ieq1}) is still an open problem.
\newpage

The Chebyshev polynomials are a sequence of orthogonal polynomials that are related to De Moivre's formula and which can be defined recursively. They have abundant properties, which make them useful in many areas in applied mathematics, numerical analysis and approximation theory. There are four kinds of Chebyshev polynomials, see for details Doha \cite{Doha} and Mason \cite{Mason}. The Chebyshev polynomials of degree $n$ of the second kind, which are denoted $U_{n}(t)$, are defined for $t\in \lbrack -1,1]$ by the following three-terms recurrence relation:
\begin{eqnarray*}
&& U_0(t)=1, \nonumber \\
&& U_1(t)=2t, \nonumber \\
&&U_{n+1}(t):=2tU_n(t)-U_{n-1}(t).
\end{eqnarray*}

The first few of the Chebyshev polynomials of the second kind are
\begin{equation} \label{ieq5}
U_{2}(t)=4t^{2}-1, \ U_{3}(t)=8t^{3}-4t, \ U_{4}(t)=16t^{4}-12t^{2}+1,\cdots.
\end{equation}

The generating function for the Chebyshev polynomials of the second kind, $U_{n}(t)$, is given by:
\begin{equation*}
H(z, t) = \frac{1}{1-2tz+z^{2}}=\sum\limits_{n=0}^{\infty }U_{n}(t)z^{n} \ \ \ (z\in
\mathbb{U}).
\end{equation*}

Yousef et al. \cite{Feras} introduced the following class $\mathscr{B}_{\Sigma }^{\mu }(\beta,\lambda,\delta)$ of analytic and bi-univalent functions defined as follows:

\begin{de} \label{def22}
For $\lambda \geq 1,\mu \geq 0, \delta \geq 0$ and $0 \leq \beta < 1$, a function $f\in \Sigma $ given by (\ref{ieq1}) is said to be in the class $\mathscr{B}_{\Sigma }^{\mu }(\beta, \lambda ,\delta)$ if the following conditions hold for all $z,w\in \mathbb{U}$:
\begin{equation} \label{ieq23}
\mbox{Re}\left((1-\lambda )\left(\frac{f(z)}{z}\right)^{\mu}+\lambda f^{\prime }(z)\left(\frac{f(z)}{z}\right)^{\mu -1}+\xi\delta zf^{\prime \prime }(z)\right)>\beta
\end{equation}
and
\begin{equation} \label{ieq24}
\mbox{Re} \left((1-\lambda)\left(\frac{g(w)}{w}\right)^{\mu }+\lambda g^{\prime }(w)\left(\frac{g(w)}{w}\right)^{\mu -1}+\xi\delta wg^{\prime \prime }(w)\right)>\beta,
\end{equation}
where the function $g(w)=f^{-1}(w)$ is defined by (\ref{ieq2}) and $\xi=\frac{2\lambda +\mu }{2\lambda +1}$.
\end{de}

This work is concerned with the coefficient bounds for the Taylor-Maclaurin coefficients $|a_2|$ and $|a_3|$ and the Fekete-Szeg\"{o} inequality for functions belonging to the class $\mathscr{B}_{\Sigma }^{\mu }(\lambda ,\delta ,t)$ defined as follows:

\begin{de} \label{def221}
For $\lambda \geq 1,\mu \geq 0, \delta \geq 0$ and $t\in \left(\frac{1}{2},1\right) $, a function $%
f\in \Sigma $ given by (\ref{ieq1}) is said to be in the class $\mathscr{B}_{\Sigma }^{\mu }(\lambda ,\delta ,t)$ if the following
subordinations hold for all $z,w\in \mathbb{U}$:
\begin{equation} \label{ieq3}
\left( 1-\lambda \right) \left(\frac{f(z)}{z}\right)^{\mu }+\lambda f^{\prime }(z)\left(\frac{f(z)}{z}\right)^{\mu -1}+\xi \delta z f^{\prime \prime }(z)\prec H(z,t):=\frac{1}{1-2tz+z^{2}}
\end{equation}
and
\begin{equation} \label{ieq4}
\left( 1-\lambda \right) \left(\frac{g(w)}{w}\right)^{\mu }+\lambda g^{\prime }(w)\left(\frac{g(w)}{w}\right)^{\mu -1}+\xi\delta w g^{\prime \prime }(w)\prec H(w,t):=\frac{1}{1-2tw+w^{2}},
\end{equation}
where the function $g(w)=f^{-1}(w)$ is defined by (\ref{ieq2}) and $\xi=\frac{2\lambda +\mu }{2\lambda +1}$.
\end{de}

The following special cases of Definitions \ref{def221} are worthy of note:

\begin{re}
Note that for $\lambda=1, \mu=1$ and $\delta =0$, the class of functions $\mathscr{B}_{\Sigma }^{1}(1,0,t):=\mathscr{B}_{\Sigma }(t)$ have been introduced and studied by Altinkaya and Yal\c{c}in \cite{F0}, for $\mu=1$ and $\delta =0$, the class of functions $\mathscr{B}_{\Sigma }^{1}(\lambda,0,t):=\mathscr{B}_{\Sigma }(\lambda,t)$ have been introduced and studied by Bulut et al. \cite{Bul}, for $\delta=0$, the class of functions $\mathscr{B}_{\Sigma }^{\mu}(\lambda,0,t):=\mathscr{B}_{\Sigma }^{\mu}(\lambda,t)$ have been introduced and studied by Bulut et al. \cite{M}, and for $\mu=1$, the class of functions $\mathscr{B}_{\Sigma }^{1}(\lambda,\delta,t):=\mathscr{B}_{\Sigma }(\lambda,\delta,t)$ have been introduced and studied by Yousef et al. \cite{C3}.
\end{re}

\section{Coefficient bounds for the function class $\mathscr{B}_{\Sigma }^{\mu }(\lambda ,\delta ,t)$}

In this section, we establish coefficient bounds for the Taylor-Maclaurin coefficients $|a_2|$ and $|a_3|$ of the function $f \in \mathscr{B}_{\Sigma }^{\mu }(\lambda ,\delta ,t)$. Several corollaries of the main result are
also considered.

\begin{thm}
\label{thm1} Let the function $f(z)$  given by (\ref{ieq1}) be in the class $\mathscr{B}_{\Sigma }^{\mu }(\lambda ,\delta ,t)$. Then
\begin{equation} \label{theq1}
|a_{2}|\leq \frac{2t\sqrt{2t}}{\sqrt{|(\lambda+\mu+2\xi\delta)^{2}-2[2(\lambda+\mu +2\xi\delta)^{2}- (2\lambda + \mu)(\mu+1)-12 \xi\delta]t^{2}|}}
\end{equation}
and
\begin{equation} \label{theq2}
\hspace{-2.2in} |a_{3}|\leq \frac{4t^{2}}{(\lambda + \mu +2\xi\delta )^{2}}+\frac{2t}{2\lambda+\mu +6\xi\delta}.
\end{equation}

\end{thm}

\begin{proof}
Let $f\in \mathscr{B}_{\Sigma }^{\mu }(\lambda ,\delta ,t)$. From (\ref{ieq3}) and (\ref{ieq4}), we have
\begin{equation} \label{eq1}
\left( 1-\lambda \right) \left(\frac{f(z)}{z}\right)^{\mu }+\lambda f^{\prime }(z)\left(\frac{f(z)}{z}\right)^{\mu -1}+\xi\delta z f^{\prime \prime }(z)= 1+U_{1}(t)w(z)+U_{2}(t)w^{2}(z)+\cdots
\end{equation}
and
\begin{equation} \label{eq2}
\left( 1-\lambda \right) \left(\frac{g(w)}{w}\right)^{\mu }+\lambda g^{\prime }(w)\left(\frac{g(w)}{w}\right)^{\mu -1}+\xi\delta w g^{\prime \prime }(w)= 1+U_{1}(t)v(w)+U_{2}(t)v^{2}(w)+\cdots,
\end{equation}
for some analytic functions
\begin{equation*} \label{eq3}
w(z)=c_{1}z+c_{2}z^{2}+c_{3}z^{3}+\cdots \qquad (z\in \mathbb{U}),
\end{equation*}
and
\begin{equation*} \label{eq4}
v(w)=d_{1}w+d_{2}w^{2}+d_{3}w^{3}+\cdots \quad (w\in \mathbb{U}),
\end{equation*}
such that $w(0)=v(0)=0,$ $|w(z)|<1$ $(z\in \mathbb{U})$ and $|v(w)|<1$ $\ (w\in \mathbb{U}).$

It follows from (\ref{eq1}) and (\ref{eq2}) that
\begin{equation*} \label{eq6}
\hspace{-.1in} \left( 1-\lambda \right) \left(\frac{f(z)}{z}\right)^{\mu }+\lambda f^{\prime }(z)\left(\frac{f(z)}{z}\right)^{\mu -1}+\xi\delta z f^{\prime \prime }(z)=
1+U_{1}(t)c_{1}z+\left[U_{1}(t)c_{2}+U_{2}(t)c_{1}^{2}\right] z^{2}+\cdots
\end{equation*}
and
\begin{equation*} \label{eq7}
\hspace{-.2in} \left( 1-\lambda \right) \left(\frac{g(w)}{w}\right)^{\mu }+\lambda g^{\prime }(w)\left(\frac{g(w)}{w}\right)^{\mu -1}+\xi\delta w g^{\prime \prime }(w)=
1+U_{1}(t)d_{1}w+\left[U_{1}(t)d_{2}+U_{2}(t)d_{1}^{2}\right] )w^{2}+\cdots.
\end{equation*}
\newpage
Equating the coefficients yields
\begin{equation} \label{eq8}
\hspace{1in} \left(\lambda+\mu + 2\xi\delta \right)a_{2}=U_{1}(t)c_{1},
\end{equation}
\begin{equation} \label{eq9}
\hspace{0.25in} (2\lambda+\mu)\left[\left(\frac{\mu -1}{2}\right)a_{2}^{2}+\left(1+\frac{6\delta }{2\lambda +1}\right)a_{3}\right]=U_{1}(t)c_{2}+U_{2}(t)c_{1}^{2},
\end{equation}
and
\begin{equation} \label{eq10}
\hspace{0.95in} -\left(\lambda+\mu +2\xi\delta \right)a_{2}=U_{1}(t)d_{1},
\end{equation}
\begin{equation} \label{eq11}
\hspace{-.3in} (2\lambda+\mu)\left[\left(\frac{\mu +3}{2}+\frac{12\delta }{2\lambda +1}\right)a_{2}^{2}-\left(1+\frac{6\delta }{2\lambda +1}\right)a_{3}\right]=U_{1}(t)d_{2}+U_{2}(t)d_{1}^{2}.
\end{equation}

From (\ref{eq8}) and (\ref{eq10}), we obtain
\begin{equation} \label{eq12}
c_{1}=-d_{1},
\end{equation}
and
\begin{equation} \label{eq13}
\hspace{-0.4in} 2\left(\lambda +\mu +2\xi\delta\right)^{2} a_{2}^2=U_{1}^2(t)\left(c_{1}^2+d_{1}^2\right).
\end{equation}

By adding (\ref{eq9}) to (\ref{eq11}), we get
\begin{equation} \label{eq14}
(2\lambda +\mu)\left[1+\mu +\frac{12\delta }{2\lambda +1}\right] a_{2}^{2}=U_{1}(t)\left(c_{2}+d_{2}\right) +U_{2}(t)\left( c_{1}^{2}+d_{1}^{2}\right).
\end{equation}

By using (\ref{eq13}) in (\ref{eq14}), we obtain
\begin{equation} \label{eq15}
\left[ (2\lambda+\mu)(\mu +1)+12\xi\delta-\frac{2U_{2}(t)\left(\lambda+\mu +2\xi\delta \right)^{2}}{U_{1}^{2}(t)}\right] a_{2}^{2}=U_{1}(t)\left(c_{2}+d_{2}\right).
\end{equation}

It is fairly well known \cite{Duren} that if $|w(z)|<1$ and $|v(w)|<1,$ then
\begin{equation} \label{eq5}
|c_{j}|\leq 1  \ \text{and} \ |d_{j}|\leq 1 \ \text{for all} \ j\in \mathbb{N}.
\end{equation}

By considering (\ref{ieq5}) and (\ref{eq5}), we get from (\ref{eq15}) the desired inequality (\ref{theq1}).

Next, by subtracting (\ref{eq11}) from (\ref{eq9}), we have
\begin{equation} \label{eq16}
2(2\lambda +\mu)\left(1+\frac{6\delta }{2\lambda +1}\right)a_{3}-2(2\lambda +\mu)\left(1+\frac{6\delta }{2\lambda +1}\right)a_{2}^{2}=U_{1}(t)\left( c_{2}-d_{2}\right) +U_{2}(t)\left(c_{1}^{2}-d_{1}^{2}\right).
\end{equation}

Further, in view of (\ref{eq12}), it follows from (\ref{eq16}) that
\begin{equation} \label{eq17}
a_{3}=a_{2}^{2}+\frac{U_{1}(t)}{2(2\lambda +\mu+6\xi\delta)}\left(c_{2}-d_{2}\right).
\end{equation}

By considering (\ref{eq13}) and (\ref{eq5}), we get from (\ref{eq17}) the desired inequality (\ref{theq2}). This completes the proof of Theorem \ref{thm1}.
\end{proof}

Taking $\lambda =1$, $\mu =1$ and $\delta = 0$ in Theorem \ref{thm1}, we get the following consequence.

\begin{co} \label{cor1} \cite{Bul}
Let the function $f(z)$ given by (\ref{ieq1}) be in the class $\mathscr{B}_{\Sigma }(t)$. Then
\begin{equation*}
|a_{2}|\leq \frac{t\sqrt{2t}}{\sqrt{1-t^{2}}},
\end{equation*}
and
\begin{equation*}
|a_{3}|\leq t^{2}+\frac{2}{3}t.
\end{equation*}
\end{co}

Taking $\mu =1$ and $\delta = 0$ in Theorem \ref{thm1}, we get the following consequence.

\begin{co} \cite{Bul}
Let the function $f(z)$  given by (\ref{ieq1}) be in the class $\mathscr{B}_{\Sigma }(\lambda ,t)$. Then
\begin{equation*}
|a_{2}|\leq \frac{2t\sqrt{2t}}{\sqrt{|(\lambda+1)^2-4\lambda^2 t^{2}|}}
\end{equation*}
and
\begin{equation*}
 |a_{3}|\leq \frac{4t^{2}}{(\lambda+1)^2}+\frac{2t}{2\lambda+1}.
\end{equation*}
\end{co}

Taking $\delta = 0$ in Theorem \ref{thm1}, we get the following consequence.

\begin{co} \cite{M}
Let the function $f(z)$  given by (\ref{ieq1}) be in the class $\mathscr{B}_{\Sigma }^{\mu }(\lambda ,t)$. Then
\begin{equation*}
|a_{2}|\leq \frac{2t\sqrt{2t}}{\sqrt{|(\lambda+\mu)^{2}-2[2(\lambda+\mu )^{2}- (2\lambda + \mu)(\mu+1)]t^{2}|}}
\end{equation*}
and
\begin{equation*}
 \hspace{-1.7in}|a_{3}|\leq \frac{4t^{2}}{(\lambda + \mu )^{2}}+\frac{2t}{2\lambda+\mu}.
\end{equation*}
\end{co}

Taking $\mu =1$ in Theorem \ref{thm1}, we get the following consequence.

\begin{co} \cite{C3}
Let the function $f(z)$  given by (\ref{ieq1}) be in the class $\mathscr{B}_{\Sigma }(\lambda ,\delta ,t)$. Then
\begin{equation*}
|a_{2}|\leq \frac{2t\sqrt{2t}}{\sqrt{\left|(1+ \lambda+2\delta)^2-4\left[(\lambda+2\delta)^2-2\delta\right]t^{2}\right|}}
\end{equation*}
and
\begin{equation*}
 \hspace{-.7in} |a_{3}|\leq \frac{4t^{2}}{(1+ \lambda+2\delta)^2}+\frac{2t}{1+2\lambda+6\delta}.
\end{equation*}
\end{co}

\section{Fekete-Szeg\"{o} problem for the function class $\mathscr{B}_{\Sigma }^{\mu }(\lambda ,\delta ,t)$}
Now, we are ready to find the sharp bounds of Fekete-Szeg\"{o} functional $a_{3}-\eta a_{2}^{2}$ defined for $\mathscr{B}_{\Sigma }^{\mu }(\lambda ,\delta ,t)$ given by (\ref{ieq1}). The results presented in this section improve or generalize the earlier results of Bulut et al. \cite{M}, Yousef et al. \cite{C3}, and other authors in terms of the ranges of the parameter under consideration.

\begin{thm}
\label{thm2} Let the function $f(z)$  given by (\ref{ieq1}) be in the class $\mathscr{B}_{\Sigma }^{\mu }(\lambda ,\delta ,t)$. Then for some $\eta \in \mathbb{R}$,
\begin{equation} \label{theq3}
|a_{3}-\eta a_{2}^{2}|\leq \left\{
\begin{array}{c}
\qquad \ \ \frac{2t}{2\lambda +\mu+6\xi\delta}, \qquad \ \qquad \qquad \qquad \qquad \quad \ \ |\eta -1|\leq M \bigskip \\
\frac{8|\eta -1|t^{3}}{\left|(\lambda+\mu+2\xi\delta)^2-2\left[2(\lambda+\mu+2\xi\delta)^2-((2\lambda+\mu)(\mu+1)+12\xi\delta)\right]t^{2}\right|}, \ \ |\eta -1|\geq M%
\end{array}%
\right.
\end{equation}
where
\begin{equation*}
M:= \frac{\left|(\lambda+\mu+2\xi\delta)^2-2\left[2(\lambda+\mu+2\xi\delta)^2-((2\lambda+\mu)(\mu+1)+12\xi\delta)\right]t^{2}\right|}{4(2\lambda+\mu+2\xi\delta)t^{2}}.
\end{equation*}

\end{thm}

\begin{proof}

Let $f\in \mathscr{B}_{\Sigma }^{\mu }(\lambda ,\delta ,t)$. By using (\ref{eq15}) and (\ref{eq17}) for some $\eta \in \mathbb{R} $, we get
\begin{equation*}
 \hspace{-0.15in} a_{3}-\eta a_{2}^{2}=\left( 1-\eta \right) \left[ \frac{U_{1}^{3}(t)\left(c_{2}+d_{2}\right) }{\left((2\lambda+\mu)(\mu+1)+12\xi\delta\right) U_{1}^{2}(t)-2(\lambda +\mu+2\xi\delta)^{2}U_{2}(t)}\right] +\frac{U_{1}(t)\left( c_{2}-d_{2}\right) }{2(2\lambda +\mu+6\xi\delta)}
\end{equation*}
\begin{equation*}
\hspace{-0.05in} =U_{1}(t)\left[ \left( h(\eta )+\frac{1}{2(2\lambda +\mu+6\xi\delta)}\right)
c_{2}+\left( h(\eta )-\frac{1}{2(2\lambda +\mu+6\xi\delta)}\right) d_2\right],
\end{equation*}
where
\begin{equation*}
 h(\eta )=\frac{U_{1}^{2}(t)\left( 1-\eta \right) }{\left((2\lambda+\mu)(\mu+1)+12\xi\delta\right) U_{1}^{2}(t)-2(\lambda +\mu+2\xi\delta)^{2}U_{2}(t)}.
\end{equation*}

Then, in view of (\ref{ieq5}), we easily conclude that
\begin{equation*}
|a_{3}-\eta a_{2}^{2}|\leq \left\{
\begin{array}{c}
\frac{2t}{2\lambda +\mu+6\xi\delta}, \ \ |h(\eta )|\leq \frac{1}{2\left( 2\lambda +\mu+6\xi\delta \right) } \bigskip \\
 4|h(\eta )|t, \ \ |h(\eta )|\geq \frac{1}{2\left( 2\lambda +\mu+6\xi\delta \right) }%
\end{array}%
\right.
\end{equation*}

This proves Theorem \ref{thm2}.
\end{proof}

We end this section with some corollaries concerning the sharp bounds of Fekete-Szeg\"{o} functional $a_{3}-\eta a_{2}^{2}$ defined for $f \in  \mathscr{B}_{\Sigma }^{\mu }(\lambda ,\delta ,t)$ given by (\ref{ieq1}).

Taking $\eta =1$ in Theorem \ref{thm2}, we get the following corollary.

\begin{co}
Let the function $f(z)$ given by (\ref{ieq1}) be in the class $\mathscr{B}_{\Sigma }^{\mu }(\lambda ,\delta ,t)$. Then
\begin{equation*}
|a_{3}-a_{2}^{2}|\leq \frac{2t}{2\lambda +\mu+6\xi\delta}.
\end{equation*}

\end{co}

Taking $\lambda =1$, $\mu =1$ and $\delta = 0$ in Theorem \ref{thm2}, we get the following corollary.

\begin{co} \label{corol2} \cite{M}
Let the function $f(z)$ given by (\ref{ieq1}) be in the class $\mathscr{B}_{\Sigma }(t)$. Then for some $\eta \in \mathbb{R}$,
\begin{equation*}
|a_{3}-\eta a_{2}^{2}|\leq \left\{
\begin{array}{c}
\frac{2}{3}t,\qquad \ |\eta -1|\leq \frac{1-t^{2}}{3t^{2}} \bigskip \\
\frac{2|\eta -1|t^{3}}{1-t^{2}}, \ \ |\eta -1|\geq \frac{1-t^{2}}{3t^{2}}%
\end{array}%
\right.
\end{equation*}

\end{co}

Taking $\eta =1$ in Corollary \ref{corol2}, we get the following corollary.

\begin{co} \cite{C3} 
Let the function $f(z)$ given be (\ref{ieq1}) be in the class $\mathscr{B}_{\Sigma}\left( t\right)$. Then
\begin{equation*}
|a_{3}-a_{2}^{2}|\leq \frac{2}{3} t.
\end{equation*}

\end{co}

Taking $\mu =1$ and $\delta = 0$ in Theorem \ref{thm2}, we get the following corollary.

\begin{co} \label{corol222} \cite{M}
Let the function $f(z)$  given by (\ref{ieq1}) be in the class $\mathscr{B}_{\Sigma}\left( \lambda ,t\right)$. Then for some $\eta \in \mathbb{R}$,
\begin{equation}
|a_{3}-\eta a_{2}^{2}|\leq \left\{
\begin{array}{c}
 \frac{2t}{1+2\lambda}, \quad \qquad \ \ |\eta -1|\leq \frac{\left|(1+\lambda)^2-4\lambda^2t^{2}\right|}{4(1+2\lambda)t^{2}} \bigskip \\
\frac{8|\eta -1|t^{3}}{\left|(1+\lambda)^2-4\lambda^2t^{2}\right|}, \ \ \ |\eta -1|\geq \frac{\left|(1+\lambda)^2-4\lambda^2t^{2} \right|}{4(1+2\lambda)t^{2}}%
\end{array}%
\right.
\end{equation}
\end{co}

Taking $\eta =1$ in Corollary \ref{corol222}, we get the following corollary.

\begin{co} \cite{C3} 
Let the function $f(z)$ given by (\ref{ieq1}) be in the class $\mathscr{B}_{\Sigma}\left( \lambda ,t\right)$. Then
\begin{equation*}
|a_{3}-a_{2}^{2}|\leq \frac{2t}{1+2\lambda}.
\end{equation*}
\end{co}

Taking $\delta = 0$ in Theorem \ref{thm2}, we get the following corollary.

\begin{co} \label{corol2222} \cite{M}
Let the function $f(z)$  given by (\ref{ieq1}) be in the class $\mathscr{B}_{\Sigma }^{\mu }(\lambda,t)$. Then for some $\eta \in \mathbb{R}$,
\begin{equation}
|a_{3}-\eta a_{2}^{2}|\leq \left\{
\begin{array}{c}
\qquad \ \ \frac{2t}{2\lambda +\mu}, \qquad \ \qquad \qquad \quad \ |\eta -1|\leq \frac{\left|(\lambda+\mu)^2-2\left[2(\lambda+\mu)^2-(2\lambda+\mu)(\mu+1)\right]t^{2}\right|}{4(2\lambda+\mu)t^{2}} \bigskip \\
\frac{8|\eta -1|t^{3}}{\left|(\lambda+\mu)^2-2\left[2(\lambda+\mu)^2-(2\lambda+\mu)(\mu+1)\right]t^{2}\right|}, \ \ |\eta -1|\geq \frac{\left|(\lambda+\mu)^2-2\left[2(\lambda+\mu)^2-(2\lambda+\mu)(\mu+1)\right]t^{2}\right|}{4(2\lambda+\mu)t^{2}}%
\end{array}%
\right.
\end{equation}
\end{co}

Taking $\mu =1$ in Theorem \ref{thm2}, we get the following corollary.

\begin{co} \label{corol22222} \cite{C3}
Let the function $f(z)$  given by (\ref{ieq1}) be in the class $\mathscr{B}_{\Sigma}\left( \lambda ,\delta ,t\right)$. Then for some $\eta \in \mathbb{R}$,
\begin{equation}
|a_{3}-\eta a_{2}^{2}|\leq \left\{
\begin{array}{c}
\qquad \ \ \frac{2t}{1+2\lambda+6\delta}, \qquad \ \qquad  |\eta -1|\leq \frac{\left|(1+\lambda+2\delta)^2-4\left[(\lambda+2\delta)^2-2\delta\right]t^{2}\right|}{4(1+2\lambda+6\delta)t^{2}} \bigskip \\
\frac{8|\eta -1|t^{3}}{\left|(1+\lambda+2\delta)^2-4\left[(\lambda+2\delta)^2-2\delta\right]t^{2}\right|}, \ \ |\eta -1|\geq \frac{\left|(1+\lambda+2\delta)^2-4\left[(\lambda+2\delta)^2-2\delta\right]t^{2}\right|}{4(1+2\lambda+6\delta)t^{2}}%
\end{array}%
\right.
\end{equation}
\end{co}




\begin{thebibliography}{10}

\bibitem {F1}
\c{S}.~Altinkaya, S.~Yal\c{c}in, Initial coefficient bounds for a general class of bi-univalent functions, Int. J. Ana., Article ID 867871,(2014), 4 pp.

\bibitem {F0}
\c{S}.~Altinkaya, S.~Yal\c{c}in, Estimates on coefficients of a general subclass of bi-univalent functions associated with symmetric q-derivative operator by means of the Chebyshev polynomials, Asia Pacific Journal of Mathematics, 4.2 (2017) 90-99.

\bibitem{L}
O.~Alt{\i}nta\c{s}, H.~Irmak, S.~Owa, H.M.~Srivastava, Coefficient bounds for some families of starlike and convex functions of complex order, Appl. Math. Lett. 20.12 (2007) 1218–-1222.

\bibitem{C1}
A.A.~Amourah, F.~Yousef, T.~Al-Hawary, M.~Darus, A certain fractional derivative operator for p-valent functions and new class of analytic functions with negative coefficients, Far East Journal of Mathematical Sciences 99.1 (2016) 75-87.

\bibitem{C4}
A.A.~Amourah, F.~Yousef, T.~Al-Hawary, M.~Darus, On a class of p-valent non-Bazilevi\u{c} functions of order $\mu+ i\beta$, International Journal of Mathematical Analysis 10.15 (2016) 701-710.

\bibitem{C5}
A.A.~Amourah, F.~Yousef, T.~Al-Hawary, M.~Darus, On H$_3$(p) Hankel determinant for certain subclass of p-valent functions, Ital. J. Pure Appl. Math 37 (2017) 611-618.

\bibitem{C6}
T.~Al-Hawary, B.A.~Frasin, F.~Yousef, Coefficients estimates for certain classes of analytic functions of complex order, Afrika Matematika (2018) 1--7.

\bibitem{C26}
D.A.~Brannan, J.G.~Clunie, W.E.~Kirwan, Coefficient estimates for a class of star-like functions, Canad. J. Math. 22 (1970) 476-485.

\bibitem{C22}
D.A.~Brannan, J.G.~Clunie, Aspects of contemporary complex analysis (Proceedings of the NATO Advanced Study Institute held at the University of Durham, Durham; July 1–20, 1979), Academic Press, New York and London, 1980.

\bibitem {Bul}
S.~Bulut, N.~Magesh, V.K.~Balaji, Initial bounds for analytic and bi–univalent functions by means of chebyshev polynomials, Analysis 11.1 (2017) 83--89.

\bibitem {M}
S.~Bulut, N.~Magesh, C.~Abirami, A comprehensive class of analytic bi-univalent functions by means of Chebyshev polynomials, J. Fract. Calc. Appl 8.2 (2017) 32--39.

\bibitem {Doha}
E.H.~Doha, The first and second kind Chebyshev coefficients of the moments of the general-order derivative of an infinitely differentiable function, Int. J. of Comput. Math. 51 (1994) 21-–35.

\bibitem{C25}
D.A.~Brannan, D.L.~Tan, On some classes of bi-univalent functions, Studia Univ. Babecs-Bolyai Math. 31.2 (1986) 70--77.

\bibitem{C29}
M.~\c{C}a\u{g}lar, H.~Orhan, N.~Ya\u{g}mur, Coefficient bounds for new subclasses of bi-univalent functions, Filomat 27.7 (2013) 1165-1171.

\bibitem{Duren}
P.L.~Duren, Univalent functions, Grundlehren der Mathematischen Wissenschaften, Band 259, Springer-Verlag, New York, Berlin, Heidelberg and
Tokyo, 1983.

\bibitem{C28}
B.A.~Frasin, M.K.~Aouf, New subclasses of bi-univalent functions, Appl. Math. Lett. 24(9) (2011) 1569--1573.

\bibitem{C21}
M.~Lewin, On a coefficient problem for bi-univalent functions, Proc. Amer. Math. Soc. 18 (1967) 63--68.

\bibitem{C212}
X.-F.~Li,  A.-P.~Wang, Two new subclasses of bi-univalent functions, Int. Math. Forum 7 (2012) 1495--1504.

\bibitem {F3}
N.~Magesh, J.~Yamini, Coefficient bounds for a certain subclass of bi-univalent functions, Int. Math. Forum, 8.27 (2013) 1337--1344.

\bibitem {Mason}
J.C.~Mason, Chebyshev polynomial approximations for the L-membrane eigenvalue problem, SIAM J. Appl. Math. 15 (1967) 172-–186.

\bibitem {mill}
S.S.~Miller, P.T.~Mocanu, Differential subordination: theory and applications, CRC Press, New York, 2000.

\bibitem{C23}
E.~Netanyahu, The minimal distance of the image boundary from the origin and the second coefficient of a univalent function in $|z|< 1$, Arch. Rational Mech. Anal. 32 (1969) 100--112.


\bibitem {F4}
S.~Porwal, M.~Darus, On a new subclass of bi-univalent functions, J. Egypt. Math. Soc., 21.3 (2013) 190--193.

\bibitem{C211}
H.M.~Srivastava, S.~Bulut, M.~\c{C}a\u{g}lar, N.~Ya\u{g}mur, Coefficient estimates for a general subclass of analytic and bi-univalent functions, Filomat 27.5 (2013) 831--842.

\bibitem{C213}
H.M.~Srivastava, S.~Gaboury, F.~Ghanim, Coefficient estimates for a general subclass of analytic and bi-univalent functions of the Ma–Minda type, Revista de la Real Academia de Ciencias Exactas, Físicas y Naturales. Serie A. Matemáticas 27.5 (2017) 1--12.

\bibitem{C27}
H.M.~Srivastava, A.K.~Mishra, P.~Gochhayat, Certain subclasses of analytic and bi-univalent functions, Appl. Math. Lett. 23.10 (2010) 1188--1192.


\bibitem{C210}
Q.-H.~Xu, Y.-C.~Gui, H.M.~Srivastava, Coefficient estimates for a certain subclass of analytic and bi-univalent functions, Appl. Math. Lett. 25 (2012) 990--994.

\bibitem{C2}
F.~Yousef, A.A.~Amourah, M.~Darus, Differential sandwich theorems for p-valent functions associated with a certain generalized differential operator and integral operator, Italian Journal of Pure and Applied Mathematics 36 (2016) 543-556.

\bibitem{Feras}
F.~Yousef, S.~Alroud, M.~Illafe, New subclasses of analytic and bi-univalent functions endowed with coefficient estimate problems, arXiv preprint arXiv:1808.06514 (2018).

\bibitem{C3}
F.~Yousef, B.A.~Frasin, T.~Al-Hawary, Fekete-Szeg\"{o} inequality for analytic and bi-univalent functions subordinate to Chebyshev polynomials, arXiv preprint arXiv:1801.09531 (2018).

\end{thebibliography}
\end{document}